\documentclass[a4paper, 10pt]{amsart}
\usepackage{amsmath}
\usepackage{amssymb, color}

\theoremstyle{plain}
\newtheorem{theorem}{\bf Theorem}[section]
\newtheorem{proposition}[theorem]{\bf Proposition}
\newtheorem{lemma}[theorem]{\bf Lemma}

\theoremstyle{definition}

\newcommand{\N}{\mathbb N}
\newcommand{\Z}{\mathbb Z}
\newcommand{\R}{\mathbb R}

 \DeclareMathOperator{\ord}{ord}

 \DeclareMathOperator{\supp}{supp}

\renewcommand{\t}{\, | \,}

\numberwithin{equation}{section}

\begin{document}

\title[The system of sets of lengths in Krull monoids  under set addition]{The system of sets of lengths in Krull monoids \\ under set addition}

\address{Institut f\"ur Mathematik und Wissenschaftliches Rechnen\\
Karl--Fran\-zens--Universit\"at Graz, NAWI Graz \\
Heinrichstra{\ss}e 36\\
8010 Graz, Austria}

\email{alfred.geroldinger@uni-graz.at}

\address{Universit{\'e} Paris 13 \\ Sorbonne Paris Cit{\'e} \\ LAGA, CNRS, UMR 7539, Universit{\'e} Paris 8 \\ F-93430, Villetaneuse \\ France} \email{schmid@math.univ-paris13.fr}

\author{Alfred Geroldinger  and Wolfgang A. Schmid}

\thanks{This work was supported by
the Austrian Science Fund FWF, Project Number P26036-N26, by the Austrian-French Amad{\'e}e Program FR03/2012, and  by the ANR Project Caesar, Project Number ANR-12-BS01-0011.}



\begin{abstract}
Let $H$ be a Krull monoid with class group $G$ and suppose that each class contains a prime divisor. Then every element $a \in H$ has a factorization into irreducible elements, and the set $\mathsf L (a)$ of all possible factorization lengths for $a$ is the set of lengths of $a$. We consider the system $\mathcal L (H) = \{ \mathsf L (a) \mid a \in H \}$ of all sets of lengths, and we characterize (in terms of the class group $G$) when $\mathcal L (H)$ is additively closed under set addition.
\end{abstract}

\maketitle

\bigskip
\section{Introduction and Main Result} \label{1}
\bigskip

By a  monoid, we mean a commutative cancellative semigroup with unit element, and we say that a monoid is atomic if  every non-unit can be written as a finite product of irreducible elements (also called atoms). Let $H$ be an atomic monoid. If $a \in H$ is a non-unit and $a = u_1 \cdot \ldots \cdot u_k$ is a factorization of $a$ into $k$ atoms, then $k$ is called the length of the factorization. The set $\mathsf L (a) \subset \N$ of all possible factorization lengths  is called the set of lengths of $a$. It is convenient to set $\mathsf L (a) = \{0\}$ for each unit $a \in H$, and we denote by $\mathcal L (H) = \{ \mathsf L (a) \mid a \in H \}$ the system of sets of lengths of $H$. All $v$-noetherian monoids (in particular, Krull monoids and the monoids of non-zero elements of noetherian domains) are atomic monoids in which all sets of lengths are finite.
Let $a, b \in H$. Then the sumset $\mathsf L (a) + \mathsf L (b)= \{l+l' \mid  l \in \mathsf L (a) , \, l' \in \mathsf L (b) \}$ is  contained in $\mathsf L (ab)$. Thus, if  $|\mathsf L (a)| > 1$ and $k \in \N$, then the $k$-fold sumset $k \mathsf L (a) = \mathsf L (a) + \ldots + \mathsf L (a)$ is contained in $\mathsf L (a^k)$, and hence $|\mathsf L (a^k)|>k$.

The system of sets of lengths $\mathcal L (H)$ is said to be {\it additively closed} if the sumset $L+L' \in \mathcal L (H)$ for all sets of lengths $L, L' \in \mathcal L (H)$.
Clearly, set addition is commutative, $\{0\} = \mathsf L (1) \in \mathcal L (H)$ is the zero-element, and it is the only invertible element. Thus $\mathcal L (H)$ is additively closed if and only if $(\mathcal L (H), +)$ is a commutative reduced semigroup with respect to set addition. Indeed, in this case it is an acyclic semigroup in the sense of \cite{Ci-Ha-Se10a}. In this paper, Cilleruelo, Hamidoune, and Serra study addition theorems in acyclic semigroups, and systems of subsets of certain semigroups with set addition as the operation belong to their main examples.

The system of sets of lengths (together with invariants controlling sets of lengths, such as elasticities and sets of distances) are the best investigated invariants in factorization theory. However, the system of sets of lengths has been explicitly determined only in some very special cases (they include Krull monoids with small class groups,  \cite[Theorem 7.3.2]{Ge-HK06a}, \cite{Ba-Ge14b}; certain numerical monoids,  \cite{A-C-H-P07a}; and  self-idealizations of principal ideal domains, \cite[Corollary 16]{Ch-Sm13a}).
Recent studies of direct-sum decompositions in module theory revealed monoids of modules  which are Krull and whose systems of sets of lengths are additively closed (\cite[Section 6.3]{Ba-Ge14b}). This phenomenon has not been observed so far in any relevant cases, and it has surprising consequences. Note that, if $H' \subset H$ is a divisor-closed submonoid, then $\mathcal L (H') \subset \mathcal L (H)$, and in all cases studied so far, a proper containment of the monoids implied a proper containment of their systems of sets of lengths. In contrast to this, suppose that   $H$ is an atomic monoid such that $\mathcal L (H)$ is additively closed. Then the direct product $H \times H$ is an atomic monoid, $H$ is a divisor-closed submonoid of $H \times H$ (up to units),  and $\mathcal L (H \times H) = \{ L+L' \mid L, L' \in \mathcal L (H) \} = \mathcal L (H)$. Proposition \ref{2.2} provides  more sophisticated consequences of the fact that a system of sets of lengths is additively closed.

Krull monoids having the property that each class contains a prime divisor have found the greatest interest in factorization theory, and they will be  the focus of the present paper. Their arithmetic can be studied with methods from Additive Combinatorics (\cite{Ge09a}). Based on a couple of recent results (see the proofs of Propositions \ref{3.1} and \ref{3.9}), we show that their systems of sets of lengths are additively closed only in a very small number of exceptional cases. Here is our main result.

\medskip
\begin{theorem} \label{1.1}
Let $H$ be a Krull monoid with class group $G$ and suppose that each class contains a prime divisor. Then the system of sets of lengths $\mathcal L (H)$ is additively closed under set addition if and only if $G$ has one of the following forms{\rm \,:}
\begin{enumerate}
\item[(a)] $G$ is cyclic of order $|G| \le 4$.
\item[(b)] $G$ is an elementary $2$-group of rank $r \le 3$.
\item[(c)] $G$ is an elementary $3$-group of rank $r \le 2$.
\item[(d)] $G$ is infinite.
\end{enumerate}
\end{theorem}

Clearly, the groups given in (a) - (c) are precisely those groups $G$ with $\exp (G) + \mathsf r (G) \le 5$.
In Section \ref{2} we outline that it is sufficient to prove Theorem \ref{1.1} for a special class of  Krull monoids and that the statement of Theorem \ref{1.1} is valid too for classes of non-Krull monoids (see Proposition \ref{2.1}). The proof of Theorem \ref{1.1} will be given in Section \ref{3}. The idea of the proof will be outlined after Proposition \ref{3.1} when we have the required concepts at our disposal.

\bigskip
\section{Context and applications} \label{2}
\bigskip

We denote by $\mathbb N$ the set of positive integers and set $\N_0 = \N \cup \{0\}$. For real number $a, b \in \R$, we denote by $[a, b] = \{ x \in \Z \mid a \le x \le b \}$ the discrete interval between $a$ and $b$. For every positive integer $n \in \N$, $C_n$ means a cyclic group of order $n$. Let $L, L' \subset \Z$ be subsets of the integers. Then $L +L' = \{ a+b \mid a \in L, b \in L' \}$ is the {\it sumset} of $L$ and $L'$. For $k \in \N$, we denote by $k L = L + \ldots + L$ the {\it $k$-fold sumset} of $L$ and by $k \cdot L = \{ k a \mid a \in L \}$ the {\it dilation} of $L$ by $k$. A positive integer $d \in \N$ is called a {\it distance} of $L$ if there exist elements $k, l \in L$ such that $k < l$, $d=l-k$, and $[k,l] \cap L = \{k,l\}$. We denote by $\Delta (L)$ the {\it set of distances} of $L$. We use the convention that $\max \emptyset = \min \emptyset = 0$.

By a {\it monoid}, we always mean a commutative semigroup with identity which satisfies the cancellation laws. If $R$ is a domain, then the multiplicative monoid $R^{\bullet} = R \setminus \{0\}$ of nonzero elements of $R$ is a monoid, and all terminology introduced for monoids will be used for domains in an obvious sense. In particular, we say that $R$ is atomic if $R^{\bullet}$ is atomic, and we set $\mathcal L (R) = \mathcal L (R^{\bullet})$ for the system of sets of lengths of $R$, and so on. A monoid $F$ is called {\it free abelian with basis} $P \subset F$ if every $a \in F$ has a unique representation of the form
\[
a = \prod_{p \in  P} p^{\mathsf v_p(a) } \quad \text{with} \quad
\mathsf v_p(a) \in \N_0 \ \text{ and } \ \mathsf v_p(a) = 0 \ \text{
for almost all } \ p \in  P \,.
\]
Let $F$ be free abelian with basis $P$. We set $F = \mathcal F( P)$ and call
\[
 |a| = \sum_{p \in  P} \mathsf v_p (a) \quad \text{the \ {\it
length}} \ \text{of} \ a \  \quad \text{and} \quad \supp (a) = \{p \in  P \mid \mathsf v_p (a) > 0 \} \quad \text{the {\it support} of} \ a \,.
\]
Clearly, $P \subset F$ is the set of primes of $F$, and if $P$ is nonempty, then, for the system of sets of lengths, we have $\mathcal L (F) = \{ \{y\} \mid y \in \N_0\}$.
A monoid $H$ is said to be a {\it Krull monoid} if it satisfies one of the following equivalent properties (\cite[Theorem 2.4.8]{Ge-HK06a} or \cite[Chapter 22]{HK98}):
\begin{enumerate}
\item[(a)] $H$ is completely integrally closed and satisfies the ascending chain condition on divisorial ideals.

\item[(b)] $H$ has a divisor homomorphism into a free abelian monoid (i.e., there is a homomorphism $\varphi \colon H \to \mathcal F (P)$ such that, for each two elements $a, b \in H$, $a$ divides $b$ in $H$  if and only if $\varphi (a)$ divides $\varphi (b)$ in $\mathcal F (P)$).
\end{enumerate}
A domain $R$ is a Krull domain if and only if $R^{\bullet}$ is a Krull monoid, and thus Property (a) shows that a noetherian domain is Krull if and only if it is integrally closed. Holomorphy rings in global fields and regular congruence monoids in these domains are Krull monoids with finite class groups such that each class contains infinitely many prime divisors (\cite[Section 2.11]{Ge-HK06a}). Monoid domains and power series domains that are Krull are discussed in \cite{Ki-Pa01, Ch11a}. For monoids of modules that are Krull we refer to \cite{Ba-Wi13a, Fa12a, Ba-Ge14b}.

We discuss a Krull monoid of a combinatorial flavor which plays a universal role in the study of sets of lengths in Krull monoids. Let $G$ be an additive abelian group. Following the tradition of combinatorial number theory (\cite{Gr13a}), the elements of $\mathcal F (G)$ will be called {\it sequences} over $G$. Let $S = g_1 \cdot \ldots \cdot g_l \in \mathcal F (G)$ be a sequence over $G$. Then  $\sigma (S) = g_1+ \ldots + g_l \in G$ is the sum of $S$, and $S$ is called a {\it zero-sum sequence} if $\sigma (S) = 0$. Clearly, the set $\mathcal B (G)$ of all zero-sum sequences over $G$ is a submonoid of $\mathcal F (G)$, and the embedding $\mathcal B (G) \hookrightarrow \mathcal F (G)$ is a divisor homomorphism. Thus $\mathcal B (G)$ is a Krull monoid by Property (b). It is easy to check that $\mathcal B (G)$ is free abelian if and only if $|G| \le 2$. Suppose that $|G| \ge 3$. Then $\mathcal B (G)$ is a Krull monoid with class group isomorphic to $G$ and each class contains precisely one prime divisor (\cite[Proposition 2.5.6]{Ge-HK06a}).

The following proposition gathers together results demonstrating the universal role of the Krull monoid $\mathcal B (G)$ in the study of sets of lengths.

\medskip
\begin{proposition} \label{2.1}~

\begin{enumerate}
\item If $H$ is a Krull monoid with class group $G$ such that each class contains a prime divisor, then $\mathcal L (H) = \mathcal L \big( \mathcal B (G) \big)$.

\item Let $\mathcal O$ be a holomorphy ring in a global field $K$, $A$ a central simple algebra over $K$, and $H$ a classical maximal $\mathcal O$-order of $A$ such that every stably free left $R$-ideal is free. Then $\mathcal L (H) = \mathcal L     \big( \mathcal B (G) \big)$, where  $G$ is a ray class group of $\mathcal O$ and hence finite abelian.

\item Let $H$ be a seminormal order in a holomorphy ring of a global field with principal order $\widehat H$ such that the natural map $\mathfrak X (\widehat H) \to \mathfrak X (H)$ is bijective and there is an isomorphism $\overline{\vartheta}\colon \mathcal{C}_v(H)\rightarrow \mathcal{C}_v(\widehat{H})$ between the $v$-class groups. Then $\mathcal L (H) = \mathcal L     \big( \mathcal B (G) \big)$, where $G = \mathcal{C}_v(H)$ is  finite abelian.
\end{enumerate}
\end{proposition}

\begin{proof}
1. See \cite[Section 3.4]{Ge-HK06a}.

2. See \cite[Theorem 1.1]{Sm13a}, and \cite{Ba-Sm15} for related results of this flavor.

3. See \cite[Theorem 5.8]{Ge-Ka-Re15a} for a more general result in the setting of weakly Krull monoids.
\end{proof}

\smallskip
Statements 2 and 3 say  that the systems of sets of lengths of the monoids under consideration coincide with the system of sets of lengths of a Krull monoid as in Theorem \ref{1.1}, and hence we know when they are additively closed. Without going into details, we would like to mention that the same is true for certain non-commutative Krull monoids (\cite{Ge13a}). Furthermore, Frisch \cite{Fr13a} showed that, for the domain $R$ of integer-valued polynomials over the integers, we have $\mathcal L (R) = \mathcal L \big( \mathcal B (G) \big)$ for an infinite group $G$.

We end this section by highlighting a surprising consequence of when the system of sets of lengths of a domain is additively closed.

\medskip
\begin{proposition} \label{2.2}
Let $R$ be an atomic domain, let $n \ge 2$ be an integer, and let $T_n (R)$ be the semigroup of upper triangular matrices with nonzero determinant. Then $\mathcal L (R) \subset \mathcal L \big( T_n (R) \big)$, and equality holds if and only if $\mathcal L (R)$ is additively closed.
\end{proposition}

\begin{proof}
Let $H = R^{\bullet}$ denote the monoid of nonzero elements of $R$.
Then \cite[Theorem 4.2]{Ba-Ba-Go14} implies that $\mathcal L \big( T_n (H) \big)$ coincides with the system of sets of lengths of the $n$-fold direct product of $H$. Therefore
\[
\mathcal L \big( T_n (H) \big) = \mathcal L (H \times \ldots \times H) =  \big\{ L_1 + \ldots + L_n \mid L_1, \ldots, L_n \in \mathcal L (H) \big\} \,,
\]
and thus the assertion follows.
\end{proof}

\bigskip
\section{Proof of Theorem \ref{1.1}} \label{3}
\bigskip

Let $G$ be an additively written finite abelian group. Then $G \cong C_{n_1} \oplus \ldots \oplus C_{n_r}$ with $1 < n_1 \t \ldots \t n_r$, where $r = \mathsf r (G) \in \N_0$ is the rank of $G$ and $n_r = \exp (G)$ is the exponent of $G$. A tuple of elements $(e_1, \ldots , e_s) \in G^s$, with $s \in \N$,  is said to be independent if $e_1, \ldots, e_s$ are non-zero and $\langle e_1, \ldots, e_s \rangle = \langle e_1 \rangle \oplus \ldots \oplus \langle e_s \rangle$. Furthermore,  $(e_1, \ldots , e_s)$ is said to be a basis of $G$ if it is independent and
$\langle e_1, \ldots, e_s \rangle = G$.

We gather the necessary concepts describing the arithmetic of  monoids of zero-sum sequences (for details and proofs, we refer to \cite{Ge-HK06a, Ge09a}).  Let $G_0 \subset G$ be a subset. Then $\mathcal B (G_0) = \mathcal B (G) \cap \mathcal F (G_0)$ denotes the submonoid of zero-sum sequences over $G_0$.  An atom of $\mathcal B (G_0)$ is a minimal zero-sum sequence over $G_0$, and we denote by $\mathcal A (G_0)$  the set of atoms of $\mathcal B (G_0)$. A sequence $S = g_1 \cdot \ldots \cdot g_l \in \mathcal F (G_0)$ is a (minimal) zero-sum sequence if and only if $-S = (-g_1) \cdot \ldots \cdot (-g_l)$ is a (minimal) zero-sum sequence. The set $\mathcal A (G_0)$ is finite and
\[
\mathsf D (G_0) = \max \{ |U| \mid U \in \mathcal A (G_0) \} \in \N
\]
is the {\it Davenport constant} of $G_0$. It is easy to see that $1 + \sum_{i=1}^r (n_i-1) \le \mathsf D (G)$. We will use without further mention that equality holds for $p$-groups and for  groups with rank $\mathsf r (G) \le 2$ (\cite[Chapter 5]{Ge-HK06a}).

\smallskip
\noindent
{\bf Factorization sets and sets of lengths.}
Let $\mathsf Z (G_0) = \mathcal F ( \mathcal A (G_0))$ denote the factorization monoid of $\mathcal B (G_0)$ (thus, $\mathsf Z (G_0)$ is the monoid of formal products of minimal zero-sum sequences over $G_0$), and let $\pi \colon \mathsf Z (G_0) \to \mathcal B (G_0)$ denote the canonical epimorphism. For $A \in \mathcal B (G_0)$, $\mathsf Z (A) = \pi^{-1} (A) \subset \mathsf Z (G_0)$ is the set of factorizations of $A$. For a factorization $z \in \mathsf Z (A)$, we call $|z| \in \N_0$ the length of $z$ and $\mathsf L (A) = \{|z| \mid z \in \mathsf Z (A)\} \subset \N_0$ is the set of lengths of $A$. Clearly, this coincides with the former informal definition. In particular, $\mathsf L (A)=\{0\}$ if and only if $A=1$, and $\mathsf L (A)=\{1\}$ if and only if $A\in \mathcal A (G_0)$. Furthermore,
\[
\mathcal L (G_0) := \mathcal L \big( \mathcal B (G_0) \big) = \{ \mathsf L (B) \mid B \in \mathcal B (G_0) \}
\]
is the system of sets of lengths of $\mathcal B (G_0)$. If $z, z' \in \mathsf Z (G_0)$ are two factorizations, say
\[
z = U_1 \cdot \ldots \cdot U_lV_1 \cdot \ldots \cdot V_m \quad \text{and} \quad z' = U_1 \cdot \ldots \cdot U_lW_1 \cdot \ldots \cdot W_n \,,
\]
where $l,m,n \in \N_0$, and all $U_i, V_j, W_k \in \mathcal A (G_0)$ with $\{V_1, \ldots, V_m\} \cap \{W_1, \ldots, W_n\} = \emptyset$, then $\mathsf d (z, z') = \max \{m,n\} \in \N_0$ is the distance between $z$ and $z'$. The distance function $\mathsf d \colon \mathsf Z (G_0) \times \mathsf Z (G_0) \to \N_0$ has the usual properties of a metric.

\smallskip
\noindent
{\bf Elasticities.} Let $|G|\ge 3$.
For $k \in \N$, we define
\[
\rho_k (G) = \max \{ \max L \mid k \in L \in \mathcal L (G) \}
\]
and recall that \cite[Section 6.3]{Ge-HK06a})
\[
\rho_{2k} (G) = k \mathsf D (G) \ , \quad  \quad 1 + k \mathsf D (G) \le \rho_{2k+1} (G) \le k \mathsf D (G) + \Big\lfloor \frac{\mathsf D (G)}{2} \Big\rfloor \,,
\]
and that
\[
\rho (G) = \max \Big\{ \frac{\max L}{\min L } \mid L \in \mathcal L (G) \Big\} = \lim_{k \to \infty} \frac{\rho_k (G)}{k} = \frac{\mathsf D (G)}{2} \,.
\]
Moreover, for  $A  \in \mathcal B (G)$, the following statements are equivalent:
\begin{itemize}
\item $\frac{\max \mathsf L (A)}{\min \mathsf L (A)} =  \frac{\mathsf D (G)}{2}$.
\item $A = (-U_1)U_1 \cdot \ldots \cdot (-U_j)U_j$ with $j \in \N$, $U_i \in \mathcal A (G)$ and $|U_i| = \mathsf D (G)$ for $i \in [1,j]$ (in which case $2j = \min \mathsf L (A)$).
\end{itemize}

\noindent
{\bf Catenary degrees.} The catenary degree $\mathsf c (A)$ of an element $A \in \mathcal B (G_0)$ is the smallest $N \in \N_0$ such that, for any two factorizations $z, z' \in \mathsf Z (A)$, there exist factorizations $z=z_0, z_1, \ldots, z_k=z'$ of $A$ such that $\mathsf d (z_{i-1}, z_i) \le N$ for each $i \in [1, k]$. Then
\[
\mathsf c (G_0) = \sup \{ \mathsf c (A) \mid A \in \mathcal B (G_0) \}
\]
denotes the catenary degree of $G_0$. It is easy to show that $\mathsf c (A) \le \max \mathsf L (A)$ and that $\mathsf c (G_0) \le \mathsf D (G_0)$.

\noindent
{\bf Sets of distances.} The set
\[
\Delta (G_0) = \bigcup_{L \in \mathcal L (G_0)} \Delta (L)
\]
is the {\it set of distances} of $\mathcal B (G_0)$. It is easy to verify that, for distinct $z,z' \in \mathsf Z (A)$, one has $d(z,z') \ge 2 + |  (|z| - |z'| ) |$. In particular, $|\mathsf Z (A)| \ge 2$ implies $2 + \max \Delta (\mathsf L (A)) \le \mathsf c (A)$, and if $\mathcal B (G_0)$ is not factorial, then $2+\max \Delta (G_0) \le \mathsf c (G_0)$. We will  further need that $\min \Delta (G_0) = \gcd \Delta (G_0)$,  and we call
\[
\Delta^* (G) = \{ \min \Delta (G_1) \mid G_1 \subset G \ \text{with} \ \Delta (G_1) \ne \emptyset \} \subset \Delta (G)
\]
the {\it set of minimal distances} of $\mathcal B (G)$. We denote by $\Delta_1 (G)$ the set of all $d \in \N$ with the following property:
\begin{itemize}
\item[]
For every $k \in \N$ there is an $L \in \mathcal L (G)$ having the following form: $L = L' \cup \{y+\nu d \mid \nu \in [0,l]\} \cup L''$, where $l \ge k$, and $L'$ and $L''$ are subsets of $L$ with $\max L' < y$ and $y+ld < \min L''$.
\end{itemize}
The relevance of the sets $\Delta^* (G)$ and $\Delta_1 (G)$ stems from their occurrence in the structure theorem for sets of lengths (see Proposition \ref{3.1} below), and it will play a crucial role in the proof of Theorem \ref{1.1}. Let $d \in \N$, \ $M \in \N_0$ \ and \ $\{0,d\} \subset \mathcal D
\subset [0,d]$. A subset $L \subset \Z$ is called an {\it almost
arithmetical multiprogression} \ ({\rm AAMP} \ for
      short) \ with \ {\it difference} \ $d$, \ {\it period} \ $\mathcal D$,
      \  and \ {\it bound} \ $M$, \ if
\[
L = y + (L' \cup L^* \cup L'') \, \subset \, y + \mathcal D + d \Z
\]
where \ $y \in \mathbb Z$ is a shift parameter,
\begin{itemize}
\item  $L^*$ is finite nonempty with $\min L^* = 0$ and $L^* =
       (\mathcal D + d \Z) \cap [0, \max L^*]$, and

\item  $L' \subset [-M, -1]$ \ and \ $L'' \subset \max L^* + [1,M]$.
\end{itemize}

\medskip
\begin{proposition} \label{3.1}
Let $G$ be a finite abelian group.
\begin{enumerate}
\item There is a constant $M \in \N_0$ such that each $L \in \mathcal L (G)$ is an {\rm AAMP}  with difference $d \in \Delta^* (G)$ and bound $M$.

\smallskip
\item $\Delta^* (G) \subset \Delta_1 (G) \subset \{ d_1 \in \Delta (G) \mid d_1 \ \text{divides some} \ d \in \Delta^* (G) \}$.

\smallskip
\item $\max \Delta^* (G) = \max \{ \exp (G)-2, \mathsf r (G)-1 \}$.
\end{enumerate}
\end{proposition}

\begin{proof}
See \cite[Corollary 4.3.16, Section 4.7]{Ge-HK06a} and \cite{Ge-Zh15a}.
\end{proof}

Note that the description in 1. is best possible by the realization theorem in \cite{Sc09a}.

\smallskip
The proof of Theorem \ref{1.1} is based on (all parts of) Proposition \ref{3.1}. We proceed in a series of propositions. The generic case is handled at the very end (in Proposition \ref{3.9}). The key idea is as follows. We choose a $d_0$ such that $L=\{2, 2+d_0\} \in \mathcal L (G)$. If $\mathcal L (G)$ would be additively closed, then the $k$-fold sumset of $L$ is in $\mathcal L (G)$ and hence $d \in \Delta_1 (G)$. Comparing the maxima of $\Delta (G)$, $\Delta_1 (G)$, and $\Delta^* (G)$, we obtain  a contradiction. Unfortunately, $\max \Delta (G)$ is known only in very special cases (even $\max \Delta (C_n \oplus C_n)$ is unknown). If $G$ is an elementary $2$-group, then $\Delta (G) = \Delta^* (G)$. Thus elementary $2$-groups need some extra care, and the same is true for elementary $3$-groups. We start with an already known case, then we handle two special groups, and after that study elementary $2$-groups (Proposition \ref{prop_el2}) and elementary $3$-groups (Proposition \ref{3.8}).

\medskip
\begin{proposition} \label{3.2}
Suppose that $G$ is cyclic. Then $\mathcal L (G)$ is additively closed if and only if $|G| \le 4$.
\end{proposition}

\begin{proof}
See \cite[Proposition 6.14]{Ba-Ge14b}.
\end{proof}

\medskip
\begin{lemma} \label{3.3}
Let $G = C_2 \oplus C_4$. Then $\mathcal L (G)$ is not additively closed.
\end{lemma}

\begin{proof}
By \cite[page 411]{Ge-HK06a}, for every $U \in \mathcal A (G)$ of length $|U|=5$, there exist $(e_1, e_2) \in G^2$ with $\ord (e_1)=2$ and $\ord (e_4)=4$ such that $U = e_2^3 e_1 (e_1+e_2)$.
Considering $U(-U)$ for such a $U$, it follows that $L = \{2,4,5\} \in \mathcal L (G)$.

We assert that the sumset $L+L = L_2 = \{4, 6,7,8,9,10\} \notin \mathcal L (G)$, which implies that $\mathcal L (G)$ is not additively closed.

We have $\mathsf D (G) = 5$ and  $\rho (G) = 5/2$. Assume to the contrary that $L_2   \in \mathcal L (G)$. Since $\max L_2 / \min L_2=5/2$ and by a result recalled in Section \ref{2}, there exist minimal zero-sum sequences $U, V \in \mathcal A (G)$ with $|U| = |V| = 5$ such that
\[
\mathsf L \big( (-U)U (-V)V \big) = L_2 \,.
\]

Let $(e_1, e_2)$ as above be given and suppose that $U = e_2^3 e_1 (e_1+e_2)$. We go through all cases for $V$ and show that $5 \in \mathsf L \big( (-U)U (-V)V \big)$, which implies the wanted contradiction. Note that $\ord (2e_2)= \ord (e_1+2e_2) = \ord (e_1) = 2$ and that $\ord (e_2) = \ord (-e_2) = \ord (e_1+e_2)=\ord (e_1-e_2) = 4$. Therefore we have
\[
\begin{aligned}
\{V \in \mathcal A (G) \mid \, |V|=5\} = \{ \, V_1 & = e_2^3e_1(e_1+e_2), \, -V_1, \\
 V_2 & = e_2^3(e_1+2e_2)(e_1-e_2), \, -V_2, \\
 V_3 & = (e_1+e_2)^3 e_1 e_2, \, -V_3, \\
 V_4 & = (e_1+e_2)^3(e_1+2e_2)(-e_2), \, -V_4 \, \} \,.
\end{aligned}
\]
Since
\[
\begin{aligned}
(-U)U (-V_1)V_1 & = \Big( (e_1+e_2)^2 e_2^2 \Big) \Big(e_2^4 \Big) \Big(e_1^2 \Big) (-U) (-U) \,, \\
(-U)U (-V_2)V_2 & = \Big( e_2^4 \Big) \Big((e_1+e_2)(e_1+2e_2)e_2 \Big) \Big( e_1(e_1-e_2)e_2 \Big) (-U)(-V_2)  \,, \\
(-U)U (-V_3)V_3 & = \Big( (e_1+e_2)^4 \Big) \Big(e_2^4 \Big) \Big( e_1^2 \Big) (-U) (-V_3) \,, \ \text{and} \\
(-U)U (-V_4)V_4 & = \Big( (e_1+e_2)^4 \Big) \Big( (e_1+2e_2)e_2^2e_1 \Big) \Big( (-e_2)e_2 \Big) (-U) (-V_4) \,,  \\
\end{aligned}
\]
it follows that $5 \in \mathsf L \big( (-U)U (-V_{\nu})V_{\nu}  \big)$ for each $\nu \in [1, 4]$.
\end{proof}

\medskip
\begin{lemma} \label{3.4}
Let $G = C_5 \oplus C_5$. Then $\mathcal L (G)$ is not additively closed.
\end{lemma}

\begin{proof}
Let $k \in \N$, $(e_1, e_2)$ be a basis of $G$, and $U = e_1^4 e_2^4 (e_1+e_2)$. Then $\mathsf L \big( (-U)U \big) = \{2,5,8,9\}$,  and we consider the $k$-fold sumset $L_k = L+ \ldots + L $. Clearly, $\min L_k = 2k$ and $\min ( L_k \setminus \{2k\}) = 2k+3$.
We assert that, for all sufficiently large $k$,  $L_k \notin \mathcal L (G)$ which implies that $\mathcal L (G)$ is not additively closed.

We have $\mathsf D (G) = 9$, $\rho (G) = 9/2$, and we set $\{U_1, -U_1, \ldots, U_s, -U_s\}  = \{ W \in \mathcal A (G) \mid |W| = 9\}$.  Let $k \in \N$ and suppose  that $L_k \in \mathcal L (G)$. Since $\max L_k/\min L_k = 9/2$ and by a result recalled in Section \ref{2},  there exist $k_1, \ldots, k_s \in \N_0$ with $k_1+ \ldots + k_s = k$ such that
\[
\mathsf L \big( (-U_1)^{k_1}U_1^{k_1}  \cdot \ldots \cdot (-U_s)^{k_s}U_s^{k_s} \big) = L_k \,.
\]
If $k$ is sufficiently large, then there is a $\nu \in [1, s]$ such that $k_{\nu} \ge 2$. We assert that $3 \in \mathsf L (U_{\nu}^2)$ for each $\nu \in [1, s]$. If this holds, then $2k+1 \in \mathsf L \big( (-U_1)^{k_1}U_1^{k_1}  \cdot \ldots \cdot (-U_s)^{k_s}U_s^{k_s} \big)$, a contradiction.

To prove the assertion, let $W \in \mathcal A (G)$ be of length $|W|=9$. By \cite[Proposition 4.2]{Ga-Ge03b} there exists a basis $(f_1, f_2)$ of $G$ such that
\[
W = f_1^4 (a_1 f_1 + f_2) (a_2 f_1 + f_2) (a_3 f_1 + f_2) (a_4 f_1 + f_2)(a_5 f_1 + f_2) \,,
\]
with $a_1, \ldots, a_5 \in [0, 4]$. Then $W^2 = \big( f_1^5 \big) S$ for some zero-sum sequence $S$ over $G$. Since $|S| = 13 > \mathsf D (G) = 9$, $S \notin \mathcal A (G)$. It follows immediately that $\mathsf L (S) = \{2\}$ and hence $3 \in \mathsf L (W^2)$.
\end{proof}

\medskip
We continue with elementary $2$-groups. Let $G=C_2^r$ with $r \ge 2$. It is well-known that $\Delta (G) = \Delta^* (G) = [1, r-1]$ (\cite[Corollary 6.8.3]{Ge-HK06a}). The next proposition summarizes our results for elementary $2$-groups.

\medskip
\begin{proposition} \label{prop_el2}
Let $G=C_2^r$ with $r \in \mathbb{N}$.
\begin{enumerate}
\item If $r=1$, then $\mathcal L (G) =  \ \big\{ \{y\} \mid y \in \N_0\big\}$. In particular, $\mathcal L (G)$ is additively closed.
\item If $r=2$, then $\mathcal L (G) = \bigl\{ y + 2k + [0,k] \, \bigm| \, y,\, k \in \N_0 \bigr\}$. In particular, $\mathcal L (G)$ is additively closed.
\item If $r=3$, then $\mathcal L (C_2^3)  =  \bigl\{ y + (k+1) + [0,k] \,
\bigm|\, y
      \in \N_0, \ k \in [0,2] \bigr\}$ \newline
      $\quad \text{\, } \ \qquad$ \quad $\cup \ \bigl\{ y + k + [0,k] \, \bigm|\, y \in \N_0, \ k \ge 3 \bigr\}
      \cup \bigl\{ y + 2k
      + 2 \cdot [0, k] \, \bigm|\, y ,\, k \in \N_0 \bigr\}$. In particular, $\mathcal L (G)$ is additively closed.
\item If $r \ge 4$, then $\mathcal{L}(G)$ is not additively closed.
\end{enumerate}
\end{proposition}

The proof of Proposition \ref{prop_el2} will be done in a series of lemmas. Since we believe that some are of interest in their own right we state them in more generality than needed for the immediate purpose at hand.  We fix our notation which will remain valid till the end of the proof of Proposition \ref{prop_el2}. Let $G=C_2^r$ with $r \in \mathbb N$ and let $(e_1, \ldots, e_r)$ be a  basis of $G$. Let $I, J \subset [1, r]$ be subsets. We denote by $I \triangle J = (I\cup J) \setminus (I \cap J)$ the symmetric difference. For an element $i \in [0,r] \setminus I$ we write $i \notin I$. If $I$ is nonempty, then we set
\[
e_I = \sum_{i \in I} e_i \,, \quad U_I = e_I\prod_{i \in I}e_i \,, \quad \text{ and} \quad V_I= e_I\prod_{i \in [0,r] \setminus I}e_i \,.
\]
Moreover, we set $e_0= e_{[1,r]}$,  $G_0 = \{e_0, \ldots, e_r\} $, and $V_0 = e_0 \cdot \ldots \cdot e_r$.
Obviously,   $\mathcal{A}(G_0) = \{h^2 \mid h \in G_0\} \cup \{ V_0\}$ and  $\mathcal{A}(G_0 \cup \{e_I\}) = \mathcal{A}(G_0)  \cup \{U_I, V_I, e_I^2\}$.

\begin{lemma} \label{3.5}
Let  $r \ge 3$.
\begin{enumerate}
\item Let $U = f_0 \cdot \ldots \cdot f_s \in \mathcal A (G)$ with $s \ge 2$.
      \begin{enumerate}
      \item The tuple $(f_1, \ldots, f_s)$ is independent and $f_0=f_1 + \ldots + f_s$.

      \smallskip
      \item If $k \in \N$,  then $\mathsf L (U^{2k}) = 2k +  (s-1) \cdot [0,k] \in \mathcal L (G)$. In particular, $\Delta ( \{f_0, \ldots, f_s \})= \{s-1\}$.
      \end{enumerate}

\smallskip
\item If $A \in \mathcal B (G)$ and $A$ is squarefree in $\mathcal F (G)$, then $\mathsf c (A) \le r$ and $\max \Delta ( \mathsf L (A)) \le r-2$.
\end{enumerate}
\end{lemma}

\begin{proof}
1.(a) \cite[Corollary 5.1.9]{Ge-HK06a} implies that $(f_1, \ldots, f_s)$ is independent. Since $U$ has sum zero, it follows that $f_0=f_1 + \ldots + f_s$.

1.(b) Let $k \in \N$. Obviously, $\mathsf L (U^2) = \{2, s+1\}$, and $U, f_0^2, \ldots, f_s^2$ are the only atoms dividing $U^{2k}$. Thus $\mathsf L (U^{2k})$ is the $k$-fold sumset of $\mathsf L (U^2)$, and hence it has the asserted form. Let $d \in \Delta ( \{f_0, \ldots, f_s \})$. Then there is a $B \in \mathcal B ( \{f_0, \ldots, f_s\})$ with $d \in \Delta ( \mathsf L (B))$. There is a $k \in \N$ such that $B \t U^{2k}$, and we set $U^{2k}=BC$ with $C \in  \mathcal B ( \{f_0, \ldots, f_s\})$. If $m \in \mathsf L (C)$, then $m + \mathsf L (B) \subset \mathsf L (U^{2k}) = 2k +  (s-1) \cdot [0,k]$, and hence $d = s-1$.

\smallskip
2. Since $\max \Delta ( \mathsf L (A)) \le \max \{0, \mathsf c (A)-2\}$, it is sufficient to prove the statement on $\mathsf c (A)$ (recall our convention that $\max \emptyset = 0$). Furthermore, it is sufficient to consider squarefree zero-sum sequences $A$ with $0 \nmid A$. We proceed by induction on $|A|$. Since $\mathsf c (A) \le \max \mathsf L (A)$, the assertion holds for all $A$ with $\max \mathsf L (A) \le r$.

Let $A$ be a squarefree zero-sum sequence with $0 \nmid A$, and let $z = U_1 \cdot \ldots \cdot U_m$ and $z' = V_1 \cdot \ldots \cdot V_n$ be two factorizations of $A$ with $m,n \in \N$ and $U_1, \ldots, U_m, V_1, \ldots, V_n \in \mathcal A (G)$. If $m \le r$ and $n \le r$, then $\mathsf d (z, z') \le r$, and we are done. So we suppose without restriction that $m > r$.

Suppose that $|V_1|= \ldots = |V_n|= \mathsf D (G)=r+1$. Since $A$ is squarefree, $\gcd_{\mathcal F (G)}(V_1, V_2) = 1$ whence $V_1V_2=W_1 \cdot \ldots \cdot W_t$ with $t \in [3,r]$, $W_1, \ldots, W_t \in \mathcal A (G)$, and $|W_1| \le r$. Since $\mathsf d ( V_1 \cdot \ldots \cdot V_n, W_1 \cdot \ldots \cdot W_tV_3 \cdot \ldots \cdot V_n) = t \le r$, we may suppose -- after a suitable change of notation -- that $|V_1| \le r$.

Let $I \subset [1,m]$ be minimal such that $V_1 \t \prod_{i \in I} U_i$, say $I = [1,l]$. Then $l \le |V_1|\le r < m$, and there are $k \in \N$ and $W_2, \ldots, W_k \in \mathcal A (G)$, such that
\[
U_1 \cdot \ldots \cdot U_m = V_1 W_2 \cdot \ldots \cdot W_kU_{l+1} \cdot \ldots \cdot U_m = V_1 \cdot \ldots \cdot V_n \,.
\]
By induction hypothesis, there are $r$-chains of factorizations from $U_1 \cdot \ldots \cdot U_{m-1}$ to $V_1 W_2 \cdot \ldots \cdot W_kU_{l+1} \cdot \ldots \cdot U_{m-1}$ and from
$W_2 \cdot \ldots \cdot W_kU_{l+1} \cdot \ldots \cdot U_m$ to $V_2 \cdot \ldots \cdot V_n$. Multiplying the first chain with $U_m$ and the second chain with $V_1$ we obtain an $r$-chain from
$U_1 \cdot \ldots \cdot U_m$ to $V_1 \cdot \ldots \cdot V_n$.
\end{proof}

We already investigated the minimal zero-sum sequences over $G_0$ and one additional element. Next we consider the problem for two additional elements.

\begin{lemma} \label{atoms}
Let $r \ge 3$ and let $I , J \subset [1,r]$  with $|I|, |J| \in [2, r-1]$.
The minimal zero-sum sequences over $G_0 \cup \{e_I,e_J\}$ which are divisible by $e_I e_J$ are
\begin{itemize}
\item $U_{I,J} = e_I e_J \prod_{i \in I \triangle J} e_i$ if $I \cap J \neq \emptyset $,
\item $V_{I,J} = e_I e_J \prod_{i \notin I \triangle J} e_i$ if both $I \not\subset J$ and  $J \not\subset I$.
\end{itemize}
\end{lemma}

\begin{proof}
Let $A \in \mathcal A (G)$ with $e_I e_J \t A$. If $I=J$, then $A=e_I^2=U_{I,I}$. Suppose that $I \ne J$. Then  $\mathsf v_{e_I}(A)= \mathsf v_{e_J} (A) = 1$.

If $e_0 \nmid A$, it follows that $A= e_I e_J \prod_{i \in I \triangle J} e_i$.  Since $A$ is neither divisible by $U_I$ nor by  $U_J$, it follows that $I \cap J \neq \emptyset$.

If $e_0 \mid A$, it follows that $A= e_I e_J \prod_{i \notin I \triangle J} e_i$. Again, any product of such a type lies in $\mathcal A (G)$ if and only if it is neither divisible by $U_I$ nor by $U_J$ (as it could only decompose as $U_IV_J$ and $U_JV_I$),  which is the case precisely when neither $I \subset J$ nor  $J \subset I$.
\end{proof}

We continue to use the notation $U_{I,J}$ and $V_{I,J}$ for all subsets $I, J \subset [1,r]$ (then $U_{I,J}$ and $V_{I,J}$ are not necessarily minimal zero-sum sequences).  

\medskip
\begin{lemma} \label{lem_length}
Let $r \ge 3$ and let $I , J \subset [1,r]$  with $|I|, |J| \in [2, r]$.
\begin{enumerate}
\item $\mathsf{L}(U_IU_J) = \{ 2,  1 + |I \cap J| \}$ if $I\cap J \neq \emptyset$, and $\mathsf{L}(U_IU_J) = \{2\}$ otherwise.

\smallskip
\item $\mathsf{L}(V_IV_J) = \{ 2,  1 + \delta + r+ 1  - |I \cup J| \}$, where $\delta =0 $ if $I\cap J \neq \emptyset$ and $\delta =1$  otherwise.

\smallskip
\item $\mathsf{L}(U_IV_J) = \{ 2,  1 + \delta + |I \setminus J| \}$, where  $\delta =0 $ if both $J \not\subset I$  and  $I \not\subset J$, and $\delta =1$  otherwise.
\end{enumerate}
\end{lemma}

\begin{proof}
1. First, we note that if there exists a factorization of $U_IU_J$ other than this one, then it must contain a minimal zero-sum sequence containing both $e_I$ and $e_J$.  We have $U_IU_J = U_{I,J}  \prod_{i \in I\cap J }e_i^2$. For $I \cap J \neq \emptyset $, we know by Lemma \ref{atoms} that $U_{I,J}$ is a minimal zero-sum sequence, and we thus have a factorization of length $1 + |I \cap J|$. If however $I\cap J= \emptyset$, then $U_{I,J} = U_I U_J$.

2. Suppose $I \cap J =  \emptyset$. Then $V_IV_J = U_I U_J \prod_{i \notin I \cup J} e_i^2$ and these two are the only factorizations not involving a minimal zero-sum sequence containing both $e_I$ and $e_J$. In this case $U_{I,J}$, is not minimal. The only remaining factorization is thus $V_{I,J} V_0$

Suppose $I \cap J \neq \emptyset$. Then  $V_IV_J$ is not divisible by  $U_{I},U_{J}$ and $V_{I,J}$, since we do not have  $e_i$ in $V_IV_J$ for $i \in I \cap J$.  The only other factorization is thus $U_{I,J} \prod_{i \notin I \cup J}e_i^2$.

3. If $J \subset I$, we observe that $V_I \mid U_I V_J $ and we get the factorization $V_IU_J \prod_{i \in I  \setminus J} e_i^2$.
The only other factorization is $U_{I,J} V_0$.

If $J \not \subset I$, we note that  $e_i$ for $i \in J \setminus I$ does not appear in $U_IV_J$. Thus, $U_IV_J$ is not divisible by $U_J$ and $U_{I,J}$. The only possibly other decomposition is thus $V_{I,J} \prod_{i \in I \setminus J}e_i^2$. Note that $V_{I,J}$ is minimal if and only if $I \not \subset J$.
\end{proof}

\medskip
\begin{lemma} \label{lem_minfact}
Let $r \ge 3$ and let $A \in \mathcal{A}(G)$ be such that  $e_I \mid A$ where $I \subset [1,r]$ with $|I| \in [2, r-1]$.
Then there exist  $B, B' \in \mathcal {B}(G) \setminus \{1\}$ with $\max\mathsf{L} (B) \le |I|$ and  $\max\mathsf{L} (B') \le r+ 1 - |I|$
such that $AV_0 = V_I B = U_I B'$.
In particular, if neither $B$ nor $B'$ is a minimal zero-sum sequence, then  $\min ( \mathsf L (AV_0) \setminus \{2\}) \le \min \{ |I| +1 , r+ 2 - |I| \} \le (r+3)/2$.
\end{lemma}

\begin{proof}
Clearly, the sequences $F= e_I^{-1}A$, $S_V=  \prod_{ i \in I}e_i$, and $S_U=  \prod_{ i \notin I}e_i$  are zero-sum free, and we
have $AV_0 = V_ I (S_V F) = U_I (S_U F)$.
We set $B= S_V F$ and $B'  = S_U F$, and  by \cite[Lemma 6.4.3]{Ge-HK06a} we infer that $\max \mathsf L (B) \le |S_V|$ and $\max \mathsf L (B') \le |S_U|$. The additional statement follows immediately.
\end{proof}

\medskip

\begin{lemma}\label{3.5_2}
Let  $r \ge 3$.
\begin{enumerate}
\item Let $A \in \mathcal B (G)$ with $\Delta (\mathsf L (A)) \neq \emptyset$ . Then the following statements are equivalent{\rm \,:}
      \begin{enumerate}
      \item $r-1 \in \Delta (\mathsf L (A))$.
      \item There is a basis $(f_1, \ldots, f_r)$ of $G$ such that $\supp (A) \setminus \{0\} = \{f_1, \ldots, f_r, f_1+ \ldots + f_r\}$.
      \end{enumerate}

\smallskip
\item Let $G_1 \subset G\setminus \{0\}$ be a subset. Then $\min \Delta (G_1) = r-1$ if and only if $G_1 = \{f_1, \ldots, f_r, f_1+ \ldots + f_r\}$ for some basis $(f_1, \ldots, f_r)$ of $G$.
\end{enumerate}
\end{lemma}

\begin{proof}
1.  Lemma \ref{3.5} shows that (b) implies (a). Conversely, let $A \in \mathcal B (G)$ such that $r-1 \in \Delta (\mathsf L (A))$, say $[l,l+r-1] \cap \mathsf L (A) = \{l , l +r-1\}$.
Since $\mathsf c (G) = r+1 $ by \cite[Theorem 6.4.7]{Ge-HK06a}, there exist factorizations $z_1$ and $z_2$ of $A$ with $|z_1|= l $ and $|z_2| = l +r-1$ such that $\mathsf d(z_1,z_2)= r+1$, say  $z_1 = U_1 \cdot \ldots \cdot U_sz$,  $z_2 = V_1 \cdot \ldots \cdot V_{t}z$ where $z = \gcd (z_1, z_2)$, $U_1, \ldots, U_s,V_1, \ldots, V_t \in \mathcal A (G)$, and $\max \{s,t\}=t=r+1$. Since $|z_1|=s+|z|=l$ and $|z_2|=t+|z|=l+r-1$, it follows that $t-s=r-1$ whence $s=2$ and $t=r+1$.  Thus $U_1U_2= V_1 \cdot \ldots \cdot V_{r+1}$, whence $U_1=U_2$,  $|U_1|=r+1$, and $|V_1|=\ldots=|V_{r+1}|=2$. Without loss of generality assume that $U_1 = V_0$.

Assume $A$ is not of the  claimed form. Then there exists some $e_I \mid A$ with $|I|\in [2, r-1]$. Let $D \mid z$ with $D \in \mathcal A (G)$ be such that $e_I \mid D$.
By Lemma \ref{lem_minfact} we have $DV_0= V_I C_V= U_I C_U$ with $C_U, C_V \in \mathcal {B}(G) \setminus \{1\}$. Since  $\max \mathsf L (D V_0)+|z| \in \mathsf L (A)$, the `in particular' statement of  Lemma \ref{lem_minfact} implies $C_U \in \mathcal {A}(G)$ or $C_V \in \mathcal {A}(G)$.

Thus, we have that $V_0 V_I C_V (D^{-1}z)$ or $V_0 U_I C_U (D^{-1}z)$  is a factorization of $A$ of length $|z_1|$. Yet, by Lemma  \ref{lem_length} it follows that $\mathsf L (V_0 V_I) = \{ 2 , 1 + (r+1 - |I|) \}$ and $\mathsf L  (V_0 U_I)= \{2, 1 + |I|\}$. Thus, $l + r -|I|$ or $l  + |I| - 1$ is an element of $\mathsf L (A)$, a contradiction.

2. That $\min \Delta( \{f_1 , \ldots, f_r, f_1+ \ldots + f_r \} )= r-1$ for a basis $(f_1, \ldots, f_r)$ follows by Lemma \ref{3.5}. Conversely, if $\min \Delta(G_1) = r-1$, then there exists some $A \in \mathcal B (G_1)$ with $r-1 \in \Delta (\mathsf L (A))$. By the first part, we get that  $\supp (A) = \{f_1, \ldots, f_r, f_1+ \ldots + f_r\}$ for a basis $(f_1, \ldots, f_r)$. If $G_1$ would contain any other element, it would equal $f_I= \sum_{i \in I} f_i$ with some $I \subset [1,r]$ and $|I | \in [2, r - 1]$. Then, $f_I \prod_{i \in I}f_i \in \mathcal A (G_1)$ and Lemma \ref{lem_length}.1 yields $|I|-1 \in  \Delta(G_1)$, a contradiction.
\end{proof}

\medskip
\begin{lemma} \label{lem_charr-2}
Let $r \ge 4$, $B \in \mathcal{B}(G)$,  and let $z_0 \in \mathsf Z (B)$ be a factorization of  length $|z_0|=\min \mathsf L (B)$ such that $V_0^2 \t z_0$.
If $\min \big( \mathsf{L}(B) \setminus \min  \mathsf{L}(B) \big) = \min\mathsf{L}(B) + ( r - 2)$, then
$| \supp (B) \setminus (G_0 \cup \{0\}) | = 1$ and this extra element is the sum of two distinct elements from $G_0$.
\end{lemma}

\begin{proof}
By Lemma \ref{3.5_2}.2,  $\supp (B) \setminus (G_0 \cup \{0\}) \ne \emptyset$, and hence there exists some $I \subset [0,r]$ such that $e_I \notin G_0$ and $e_I \mid B$.  Let $A_I \in \mathcal A (G)$ be such that $A_I \mid z_0$ and $e_I \mid A_I$.  Since $\min  \mathsf{L}(B)- 2 + \mathsf{L}(A_IV_0) \subset \mathsf{L}(B) $ and since $r > (r+3)/2$, it follows by Lemma \ref{lem_minfact} that $A_I V_0 =W_IC_I$ with $W_I \in \{U_I, V_I \}$ and $C_I \in \mathcal A (G)$.

By Lemma \ref{lem_length} we have that $\mathsf{L}(U_I V_0) = \{2, |I|+ 1\}$.  Thus if $W_I = U_I$, we infer that $|I|-1 \ge r-2$ and thus $|I|= r-1$.
We also have $\mathsf{L}(V_I V_0) = \{2, 2 + r - |I|\}$. Thus if $W_I = V_I$, we infer that $|I| = 2$.
Therefore we have shown that each non-zero element in $\supp (B) \setminus G_0$  is  the sum of two distinct elements from $G_0$.

Now, we assume to the contrary that there exist two distinct sets $I, J \subset [1,r]$  such that $e_I, e_J \notin G_0$ and  $e_I e_J \mid B$.  Let $z_0' = W_I C_I  ( (A_IV_0)^{-1}z_0 ) $ be the factorization constructed above and note that $V_0$ divides $z_0'$. Let $A_J \in \mathcal A (G)$ be such that $A_J \mid z_0'$ and $e_J \mid A_J$. Note that $A_J  \neq W_I$.
As above we obtain that $A_J V_0 $ equals $W_JC_J$ with $W_J \in \{U_J, V_J \}$, $C_J \in \mathcal A (G)$,  and $|J| \in \{2, r-1\}$.
In particular, we have a factorization $z_0'' \in \mathsf Z (B)$ of minimal length with $W_I W_J \t z_0''$ and hence $\min  \mathsf{L}(B) - 2 + \mathsf{L}(W_IW_J) \subset \mathsf{L}(B) $.

We analyze $\mathsf{L}(W_IW_J)$, and distinguish four cases. We use Lemma \ref{lem_length} throughout.

\noindent
CASE 1: \, $W_I= U_I$ and $W_J = U_ J $.

We have $|I|= |J|= r-1$  and thus  $|I  \cap J| = r-2$ as $I \neq J$. Now
$\mathsf{L}(U_IU_J) = \{2, |I  \cap J| + 1 \} = \{2, r-1\}$, a contradiction.

\smallskip
\noindent
CASE 2: \,  $W_I= U_I$ and $W_J = V_ J $.

We have $|I|= r-1$ and $|J| = 2$.
If  $J \subset I$, then $\mathsf{L}(U_IV_J) = \{2, 2 + |I  \setminus  J|  \} = \{2, r-1\}$, a contradiction.
If  $J  \not \subset I$, then $\mathsf{L}(U_IV_J) = \{2, 1 + |I  \setminus  J|  \} = \{2, r-1\}$, a contradiction.

\smallskip
\noindent
CASE 3: \,  $W_I= V_I$ and $W_J = U_ J $.

Completely analogous to CASE 2.

\smallskip
\noindent
CASE 4: \,  $W_I= V_I$ and $W_J = V_ J $.

We have  $|I|= |J| = 2$.
If  $ I \cap J = \emptyset $, then $\mathsf{L}(V_IV_J) = \{2, 2  + r +1 - |I  \cup  J|  \} = \{2, r-1\}$, a contradiction.
If  $ I \cap J \neq \emptyset $, then $\mathsf{L}(V_IV_J) = \{2, 1 + r + 1 -  |I  \cup  J|  \} = \{2, r-1\}$, a contradiction.
\end{proof}

\medskip
\begin{proof}[Proof of Proposition \ref{prop_el2}]
 For $r \le 3$ the claim follows from \cite[Theorem 7.3.2]{Ge-HK06a}. We assume $r\ge 4$ and need to show that $\mathcal{L}(G)$ is not additively closed.

By Lemma \ref{3.5}, we infer that $L' = \{4, r+ 2, 2r\} \in \mathcal L (G)$ and $L_k'' = 2k +  (r-1) \cdot [0,k] \in \mathcal L (G)$ for each $k \in \N$.
We assert that the sumset $L_k = L' + L_k'' \notin \mathcal L (G)$  for all sufficiently large $k \in \N$. Assume to the contrary that  there exist $B_k = 0^{v_k} B_k'$, where $v_k \in \N_0$ and $B_k' \in \mathcal B (G \setminus \{0\})$,  such that $\mathsf{L}(B_k)= L_k$ for each $k \in \N$.
Note that $\min L_k = 2k + 4$, $\min L_k \setminus \{2k+4\}=2k+r+2 = \min L_k + (r-2)$, and $\max L_k = k (r+1) + 2r$. We consider a factorization of minimal length and one of maximal length, say
\[
B_k = 0^{v_k}X_1 \cdot \ldots \cdot X_{2k+4-v_k} = 0^{v_k}Y_1 \cdot \ldots \cdot Y_{k(r+1)+2r-v_k}
\]
where all $X_i, Y_j \in \mathcal A (G) \setminus \{0\}$. Then
\[
v_k + 2 \big(k(r+1)+2r-v_k\big) \le v_k + \sum_{\nu=1}^{k(r+1)+2r-v_k}|Y_{\nu}| = |B_k| = v_k + \sum_{\nu=1}^{2k+4-v_k}|X_{\nu}| \le v_k+(2k+4-v_k)(r+1) \,.
\]
Since the difference between the upper and lower  bound equals $4 - v_k(r-1)$,
it  follows that $v_k\le 1$, that at most $4$ of the atoms $Y_1, \ldots, Y_{k(r+1)+2r-v_k}$ do not have length $2$, at most four of the atoms $X_1, \ldots, X_{2k+4-v_k}$ do not have length $r+1$, and thus  at least $k$ of the $X_i$ have length $r+1$.
Since $\mathcal A (G)$ is finite, it follows  that, for all sufficiently large $k$, any factorization of $B_k$ of minimal length  contains a minimal zero-sum sequence of length $r+1$ with multiplicity at least $6$.

Now suppose that $k$ is sufficiently large that this holds, and
without restriction  suppose that  $V_0$ is the atom with multiplicity $6$.
By Lemma  \ref{lem_charr-2}, $| \supp (B_k) \setminus (G_0 \cup \{0\}) | = 1$ and this additional element is the sum of two distinct elements from $G_0$.
Without restriction we may suppose that  $e_0 + e_r = \sum_{i=1}^{r-1}e_i$ is this element.
We set $I = [1,r-1]$ and assert that $\mathsf{v}_{e_I}(B_k) \in [2,4]$.

Assume to the contrary that $\mathsf{v}_{e_I}(B_k)=1$. Then $U_I$ and $V_I$ are the only minimal zero-sum sequences  containing $e_I$ that divide $B_k$. We set $B_k= U_I C_k = V_I D_k$, with $C_k, D_k \in \mathcal B (G_0)$, and obtain that  $\mathsf{Z} (B_k) = U_I \mathsf{Z} (C_k) \cup  V_I \mathsf{Z} (D_k)$. By Lemma \ref{3.5}, $\mathsf L (C_k)$ and $\mathsf L (D_k)$ are arithmetical progressions with difference $r-1$, and thus $\mathsf{L}(B_k)$ is a union of two arithmetical progression with difference $r-1$, a contradiction to $\mathsf L (B_k) = L_k$.

The only  minimal zero-sum sequences containing $e_I$ over $\supp (B_k) \subset  G_0 \cup \{0, e_I\}$ are $e_I^2$, $U_I$, and $V_I$, having lengths $2$, $r$, and $3$, respectively. If $e_I^2$ occurs, then rechecking the above chain of inequalities shows that there are at most two minimal zero-sum sequences in a factorization of minimal length that do not have length $r+1$, and hence  $\mathsf{v}_{e_I}(B_k)\le 4 $. If $e_I^2$ does not occur, then we also obtain that $\mathsf{v}_{e_I}(B_k)\le 4 $, because we know that there are at most $4$ of the minimal zero-sum sequences in a factorization of minimal length do not have length $r+1$.

Now, we assert that $\max  \mathsf{L}  (B_k) - 1  \in \mathsf{L}  (B_k)$,  a contradiction to $\max L_k - 1 \notin L_k$.
Consider a factorization  $z \in \mathsf Z (B_k)$ of maximal length $|z| = \max \mathsf L (B_k)$.
We know that  most $4$ atoms dividing $z$  do not have length $2$, and thus the atoms $e_0^2$ and $e_r^2$ divide $z$; recall that $V_0$ has multiplicity $6$ in $B_k$.
Since $\mathsf{v}_{e_I}(B_k)\in [2, 4]$ and  the only atoms containing $e_I$ over $\supp (B_k) \subset  G_0 \cup \{0, e_I\}$ are $e_I^2$, $U_I$, and $V_I$, $z$
is divisible by $e_I^2$, or by $U_I^2$, or by $V_I^2$, or by $U_IV_I$. Clearly, no factorization of maximal length is divisible by  $U_I^2$ or by $V_I^2$. Since $U_IV_I = e_I^2 V_0$, we may assume without restriction that $z$ is divisible by  the atom $e_I^2$.
Since $z$ is also divisible by $e_0^2$ and by $e_r^2$, and since $e_I^2e_0^2e_r^2 =V_I^2$, we obtain  a factorization of length $|z|-1 = \max  \mathsf{L}  (B_k) - 1$, yielding the desired contradiction.
\end{proof}

\bigskip

We continue with elementary $3$-groups.
If $r \in [1,3]$, then $\Delta (C_3^r) = \Delta^* (C_3^r) = [1, \max \{r-1, 1\}]$ (this follows from \cite[Corollary 5.1]{Ge-Gr-Sc11a}). If $r \ge 4$, then $[1, r-1] = \Delta^* (C_3^r) \subset \Delta (C_3^r)$, and it is an open problem whether equality holds or not.

\medskip
\begin{proposition} \label{3.8}
Let $G = C_3^r$ with $r \in \N$.
\begin{enumerate}
\item If $r=1$, then $\mathcal L (G) = \{y+2k+ [0,k] \mid y, k \in \N_0\}$. In particular, $\mathcal L (G)$ is additively closed.

\smallskip
\item If $r=2$, then
      \[
      \mathcal L (G) = \big\{ \{ 1 \}  \big\} \cup \big\{ [2k, \nu] \mid k \in \N_0 , \nu \in [2k, 5k] \big\} \cup \big\{ [2k+1, \nu] \mid k \in \N, \nu \in [2k+1, 5k+2] \big\} \,.
      \]
      In particular, $\mathcal L (G)$ is additively closed.

\smallskip
\item If $r \ge 3$, then $\mathcal L (G)$ is not additively closed.
\end{enumerate}
\end{proposition}

\begin{proof}
Let $r \ge 2$, $(e_1, \ldots, e_r)$ be a basis of $G$, and $U = e_1^2 \cdot \ldots \cdot e_r^2 e_0$ with $e_0 = e_1+ \ldots + e_r$. We assert that
\[
\mathsf L \big( (-U)U \big) = [2, r+2] \cup \{2r+1\} \,.
\]
Suppose that $(-U)U = V_1 \cdot \ldots \cdot V_s$ with $s \in \N$ and $V_1, \ldots, V_s \in \mathcal A (G)$. If $(-e_0)e_0 \in \{V_1, \ldots , V_s\}$, then $s=2r+1$. Otherwise, we may suppose without restriction that $e_0 \t V_s$ and $-e_0 \t V_{s-1}$. There is a subset $J \subset [1, r]$ such that
\[
V_s = e_0 \prod_{j \in J} (-e_j) \prod_{i \in I} e_i^2 \quad \text{and} \quad I = [1, r] \setminus I \,.
\]
This implies that
\[
V_{s-1} = (-e_0) \prod_{j \in J} e_j \prod_{i \in I} (-e_i)^2 = - V_s \,.
\]
Therefore we obtain that $V_1 \cdot \ldots \cdot V_{s-2} = \prod_{j \in J} \Big( (-e_j)e_j \Big)$ and hence $s = |J|+2$. Summing up we infer that
\[
\mathsf L \big( (-U)U \big) = \{2r+1\} \cup \{ 2+|J| \mid J \subset [0,r] \} = \{2r+1\} \cup [2, r+2] \,.
\]

\smallskip
1. By \cite[Theorem 7.3.2]{Ge-HK06a}, $\mathcal L (G)$ has the given form, which immediately implies that $\mathcal L (G)$ is additively closed.

\smallskip
2. Suppose that $r=2$. It is sufficient to show that $\mathcal L (G)$ has the asserted form. Then it can be verified immediately that $\mathcal L (G)$ is additively closed.

We have $\mathsf D (G)=5$, $\rho (G)=5/2$,  $\Delta (G) = \{1\}$ (\cite[Corollary 6.4.9]{Ge-HK06a}), and $\rho_k (G) = \lfloor k \mathsf D (G)/2 \rfloor$ by \cite[Theorem 6.3.4]{Ge-HK06a} for all $k \ge 2$. These facts imply that every $L \in \mathcal L (G)$ equals one of the sets given on the right hand side. So it remains to verify that conversely every set $L$ given on the right hand side can be realized as a set of lengths in $\mathcal L (G)$. Clearly, $\{k\} \in \mathcal L (G)$ for each $k \in \N_0$.  Let $k \in \N$.

First, we assert that $[2k, \nu] \in \mathcal L (G)$ for all $\nu \in [2k, 5k]$, and we proceed by induction on $k$.
The construction above shows that $[2, 5] \in \mathcal L (G)$. If $W_3 = e_1e_2(-e_0)$, then  $\mathsf L \big( (-W_3)W_3 \big) = [2,3] \in \mathcal L (G)$. If $W_4 = e_1^2e_2(e_1-e_2)$, then $\mathsf L \big( (-W_4)W_4 \big) = [2, 4] \in \mathcal L (G)$. Thus the assertion holds for $k=1$.
Suppose the assertion holds for $k \in \N$. If $\nu \in [2k, 5k]$ and $A_{\nu} \in \mathcal B (G)$ with $\mathsf L (A_{\nu}) = [2k, \nu]$, then $\mathsf L (0^2 A_{\nu}) = [2k+2, \nu+2]$. Thus it remains to show that $[2k+2, 5k+3], [2k+2, 5k+4]$, and $[2k+2, 5k+5] \in \mathcal L (G)$. If $U, W_3$, and $W_4$ are as above, then
\[
\begin{aligned}
\mathsf L \big( (-U)^k U^k \big) & = [2k, 5k] \,, \\
\mathsf L \big( (-U)^k U^k (-W_3)W_3 \big) & = [2k+2, 5k+3] \,, \\
\mathsf L \big( (-U)^k U^k (-W_4) W_4 \big) & = [2k+2, 5k+4] \,, \quad \text{and} \\
\mathsf L \big( (-U)^{k+1} U^{k+1} \big) & = [2k+2, 5k+5] \,. \\
\end{aligned}
\]

Next, we assert that $[2k+1, \nu] \in \mathcal L (G)$ for all $\nu \in [2k+1, 5k+2]$. If  $k \in \N$, $\nu \in [2k, 5k]$, and $A_{\nu} \in \mathcal B (G)$ with $\mathsf L (A_{\nu}) = [2k, \nu]$, then $\mathsf L (0 A_{\nu}) = [2k+1, \nu+1]$. Since $\rho_{2k+1} (G) = 5k+2$, there is a $B_k \in \mathcal B (G)$ with $2k+1, 5k+2 \in \mathsf L (B_k)$ and hence  $\mathsf L (B_k) = [2k+1, 5k+2] \in \mathcal L (G)$.

\smallskip
3. Suppose that $r \ge 3$. Let $k \in \N$. We consider the $k$-fold sumset $L_k = L+ \ldots + L$ of $L = \mathsf L \big( (-U)U \big)$. We assert that, for all sufficiently large $k$,  $L_k \notin \mathcal L (G)$, which implies that $\mathcal L (G)$ is not additively closed. We set $\{U_1, -U_1, \ldots, U_s, -U_s\}  = \{ W \in \mathcal A (G) \mid |W| = \mathsf D (G) \}$.  Let $k \in \N$ and suppose  that $L_k \in \mathcal L (G)$. Since $\max L_k /\min L_k$ equals $\rho(G)$ and by a result recalled in Section \ref{2}, there exist $k_1, \ldots, k_s \in \N_0$ with $k_1+ \ldots + k_s = k$ such that
\[
\mathsf L \big( (-U_1)^{k_1}U_1^{k_1}  \cdot \ldots \cdot (-U_s)^{k_s}U_s^{k_s} \big) = L_k \,.
\]
Note that $\max  L_k = k(2r+1)$ and that $\max \big( L_k \setminus \{ k(2r+1) \} \big) = k(2r+1)-(r-1)$. There is a unique factorization of length $\max L_k$. It consists entirely of atoms having length two. If $k$ is sufficiently large, then there is a $\nu \in [1, s]$ such that $k_{\nu} \ge 3$, say $\nu=1$ and $U_1 = gS$ with $g \in G$ and $S \in \mathcal F(G)$. Then the factorization of length $\max L_k$ contains the product $\big( (-g)g \big)^3$. Since
\[
\big( (-g)g \big)^3  = (g^3) \big( (-g)^3 \big) \,,
\]
it follows that $\max L_k -1 \in \mathsf L \big( (-U_1)^{k_1}U_1^{k_1}  \cdot \ldots \cdot (-U_s)^{k_s}U_s^{k_s} \big) $, a contradiction.
\end{proof}

\smallskip
Finally, we handle the generic case.

\medskip
\begin{proposition} \label{3.9}
Let $G$ be a finite abelian group with  $\exp (G) =n \ge 4$ and $r = \mathsf r (G) \ge 2$. Then $\mathcal L (G)$ is not additively closed.
\end{proposition}

\begin{proof}
If $G = C_5 \oplus C_5$ or $G= C_2 \oplus C_4$, then the assertion follows from Lemma \ref{3.3} and from Lemma \ref{3.4}. Let $G = C_{n_1} \oplus \ldots \oplus C_{n_r}$ with $1 < n_1 \t \ldots \t n_r$, $|G| \ge 5$, and suppose that $G$ is distinct from the above two groups.  Simple examples (\cite[Theorem 6.6.2]{Ge-HK06a}) show that
\[
\{2,d \} \in \mathcal L(G) \quad \text{for all} \quad d \in \bigl[ 3,\, \max\{ n,\, d_0 \} \bigr]\,, \quad \text{where} \quad d_0 = 1 + \sum_{i=1}^{\mathsf r (G)} \Bigl\lfloor \frac{n_i}{2}
\Bigr \rfloor\,.
\]
Assume to the contrary
that $\mathcal L (G)$ is additively closed.
Then $d-2 \in \Delta_1 (G)$ for each $d$ as above,  in particular $d_0 -2  \in \Delta_1 (G)$, and the interval $[1, n-2] \subset \Delta_1 (G)$.
We use that
$\max \Delta_1 (G) \le \max \Delta^* (G) = \max \{r-1, n-2\}$ by Proposition \ref{3.1}.

If $r-1 \ge n-2$, then $n \ge 4$ implies that $ d_0-2 > r-1 = \max \Delta^* (G)$, a contradiction.
Thus it follows that $r-1 < n-2$. We distinguish three cases.

\smallskip
\noindent CASE 1: \, $G=C_n \oplus C_n$.

If $n$ is even, then $d_0-2=n-1>n-2=\max \Delta^* (G)$, a contradiction. Suppose that $n$ is odd. Then $n \ge 7$ and $n-4 \in \Delta_1 (G)$. By \cite[Corollary 3.8]{Sc09c}, it follows that
\[
\max \Delta^* (C_n \oplus C_n) \setminus \{n-3, n-2\} = \frac{n-3}{2} \,.
\]
Since $n \ge 7$, it follows that $n-4 > \frac{n-3}{2}$, a contradiction.

\smallskip
\noindent CASE 2: \, $G$ has  a proper subgroup isomorphic to $C_n \oplus C_n$.

Then $d_0-2 \ge n-1 > n-2 = \max \Delta^* (G)$, a contradiction.

\smallskip
\noindent CASE 3: \, $G$ has  no subgroup isomorphic to $C_n \oplus C_n$.

Then it follows that $n_{r-1} \le n_r/2$. If $r=n-2$, then $n \ge 6$ (because $G \notin \{ C_2 \oplus C_4, C_5, C_5 \oplus C_5\}$) and thus
\[
d_0-2 \ge r-1 + \Big\lfloor \frac{n}{2} \Big\rfloor - 1 = n-4 + \Big\lfloor \frac{n}{2} \Big\rfloor > n-2 \,,
\]
a contradiction.

Suppose that $r \le n-3$. Then $n \ge 5$. If $n=5$, then $G$ is either cyclic or has a subgroup isomorphic to $C_5 \oplus C_5$, a contradiction. Thus $n \ge 6$. Then \cite[Theorem 3.2]{Sc09c} implies that
\[
\Delta^* (G) \subset [1, \max \{ \mathsf m (G), \lfloor n/2 \rfloor-1\}] \cup \{n-2\} \,, \quad \text{where}
\]
\[
\mathsf m (G) = \max \{ \min \Delta (G_0) \mid G_0 \subset G \ \text{is a non-half-factorial LCN-set}  \} \,.
\]
Since    $\mathsf m (G) < n-3$ by \cite{Ge-Zh15a}, it follows that $\max \{ \mathsf m (G), \lfloor n/2 \rfloor-1\}]  \le n-4$.
This implies that  $n-3 \notin \Delta^* (G)$, but $n-3 \in \Delta_1 (G)$, a contradiction to Proposition \ref{3.1}.2.
\end{proof}

\medskip
\begin{proof}[Proof of Theorem \ref{1.1}]
Let $H$ be a Krull monoid with class group $G$ and suppose that each class contains a prime divisor. By Proposition \ref{2.1}, it is sufficient to consider the monoid $\mathcal B (G)$ instead of the monoid $H$.

First suppose that $G$ is infinite. By the Realization Theorem of Kainrath, every finite subset $L \subset \N_{\ge 2}$ can be realized as a set of lengths in $\mathcal L (G)$. Thus we obtain that
\[
\mathcal L (G) = \{ L \subset \N_{\ge 2} \mid L \ \text{is finite and nonempty} \} \ \cup \ \{ \{0\}, \{1\} \} \,,
\]
(see \cite{Ka99a} or \cite[Theorem 7.4.1]{Ge-HK06a}), which shows that $\mathcal L (G)$ is additively closed.

Suppose now that $G$ is finite. Cyclic groups are considered in Proposition \ref{3.2}, elementary $2$-groups in Proposition \ref{prop_el2}, and elementary $3$-groups in Proposition \ref{3.8}. The case of non-cyclic groups with exponent $n \ge 4$ is settled by Proposition \ref{3.9}.
\end{proof}

\bigskip
\noindent
{\bf Acknowledgement.} We would like to thank the referee for reading the paper very carefully and for making many helpful remarks.

\bigskip

\providecommand{\bysame}{\leavevmode\hbox to3em{\hrulefill}\thinspace}
\providecommand{\MR}{\relax\ifhmode\unskip\space\fi MR }
\providecommand{\MRhref}[2]{%
  \href{http://www.ams.org/mathscinet-getitem?mr=#1}{#2}
}
\providecommand{\href}[2]{#2}

\end{document}